\RequirePackage{fix-cm}
\documentclass[smallextended]{svjour3}       
\usepackage{graphicx} 

\usepackage{booktabs} 
\usepackage{array} 
\usepackage{paralist} 
\usepackage{verbatim} 
\usepackage{algorithmic}
\usepackage{algorithm}
\usepackage{amsmath}
\usepackage{amssymb}
\usepackage{caption}
\usepackage{subcaption}
\usepackage{mathtools,amscd}

\usepackage[nottoc,notlof,notlot]{tocbibind} 
\newtheorem{assumption}{Assumption}
\newcommand{\real}{\mathbb{R}}
\newcommand{\bx}{{x}}

\newcommand{\bas}{\bar{\sigma}}
\newcommand{\bax}{\bar{x}}
\newcommand{\bal}{\bar{\lambda}}
\newcommand{\bam}{\bar{\mu}}
\newcommand{\bat}{\bar{\tau}}
\newcommand{\VV}{V}
\newcommand{\bF}{{\bf F}}
\newcommand{\GG}{{G}}
\newcommand{\sta}{\mbox{sta}}
\newcommand{\bsig}{\mbox{\boldmath$\sigma$}}

\begin{document}

\title{Canonical duality for solving general nonconvex constrained problems}
\author{Vittorio Latorre, David Yang Gao}

\institute{V. Latorre \at
              Department of Computer, Control and Management Engineering, University of Rome ÒSapienzaÓ, via Ariosto 25, Rome, Italy \\
              \email{latorre@dis.uniroma1.it}           
           \and
           D. Y. Gao \at
              School of Science, Information Technology and Engineering, University of Ballarat, Mt Helen, Ballarat, Victoria 3350\\
                            \email{d.gao@ballarat.edu.au}           
}

\date{Received: date / Accepted: date}

\maketitle

\abstract
This paper presents a canonical duality theory for solving a general nonconvex constrained optimization problem within a unified framework
to cover Lagrange multiplier method and KKT theory.
It is proved that if both target function and constraints possess certain patterns necessary for modeling real systems,
 a perfect dual problem (without duality gap)
 can be obtained in a unified form with  global optimality conditions provided.
 While the popular  augmented Lagrangian method may   produce more difficult  nonconvex problems  due to the nonlinearity of  constraints.

\section{Introduction}
We are interested in solving the following nonconvex constrained minimization problem:
\begin{equation}\label{eq: original problem}
\begin{array}{lll}
\min & f(x)\\
s.t. & g_i(x)\le0 & i=1,\dots, m\\
&       h_j(x)=0 & j=1,\dots, p,
\end{array}
\end{equation}
where $f$, $g_i$ and $h_j$ are smooth, real-valued functions  on a subset of $\real^n$ for all $i=1,\dots, m$ and
$j=1,\dots, p$. For notational convenience,
we use vector form for  constraints $g(x)$ and $h(x)$ (without the subscript):
$$
\begin{array}{l}
g(x)=\left(g_1(x),\dots, g_m(x)\right) , \\
h(x)=\left(h_1(x),\dots, h_p(x)\right).
\end{array}
$$
Therefore, the feasible space can be defined as
$$
{\cal X}_a:=\{x\in \real^n| g(x)\le0,h(x)=0 \}.
$$

 Lagrange  multiplier method  was originally proposed  by J-L. Lagrange from analytical mechanics
 in 1811 \cite{lag}.
 During the past two hundred  years, this method and the associated Lagrangian duality theory have been well-developed
 with extensively applications  to many fields of physics,  mathematics and engineering sciences.
 Strictly speaking, the Lagrange multiplier method can be used only for equilibrium constraints.
For inequality constraints, the additional  KKT conditions should be considered.
In order to solve inequality constrained problems, penalty methods and augmented Lagrangian methods have been
studied extensively during the past fifty years (see \cite{powell,bgm2005,aug2000}).
However, these well-developed methods can be used mainly for solving linear inequality constrained problems.
For nonlinear constraints,  say even  the most simple quadratic constraint $\|x\|^2 \le r$ which is essential for virtually any
real-world system \cite{chen-gao-OMS},
the (external) penalty/augmented Lagrangian methods produce
 a  nonconvex   term $\frac{1}{2}  \alpha (\|x\|^2 - r)^2_+$
in the problem.

Canonical duality theory is potentially powerful methodological method, which
was developed originally from nonconvex analysis/mechanics   \cite{GaoBook 2000,gao-jogo00}.
This theory has been used successfully for solving a large class of challenging problems in  nonconvex/nonsmooth/discrete systems
 \cite{gao-cace09,wang-etal,zgy},
 recently in network communications \cite{g-r-p,ruan-gao-ep} and radial basis neural networks \cite{LaG 13}.
 It was shown in \cite{gao-sherali-amma}
 that both the Lagrange multiplier method and KKT conditions can be unified within a framework of the canonical duality theory.
  This unified framework leads to an elegant and simple way to handle
 nonlinear constrained optimization problems. The associated triality theory
  provides  global optimal conditions which can be used to develop efficient algorithms for solving
  general nonconvex constrained problems (see \cite{gao-ruan-sherali-jogo,gao-yang}).

The canonical duality theory for solving nonconvex constrained quadratic minimization problem
has been discussed in \cite{gao-ruan-sherali-jogo}.
The main goal of this paper is to demonstrate how to use the canonical duality theory for solving the general
  non-convex constrained problem (1).

\section{Unity for  Convex Problems}
For a given convex feasible set $\cal E$, its indicator function $\Psi(\epsilon)$ is defined by
\begin{equation}
\Psi(\epsilon)=
\begin{cases}
0 \qquad \mbox{if } \epsilon \in {\cal E}\\
+\infty \quad \mbox{otherwise}.
\end{cases}
\end{equation}
The Legendre conjugate of $\Psi(\epsilon)$ is defined by using the Fenchel transformation
\begin{equation}
\Psi^*(\epsilon^*)=\sup_{\epsilon \in \cal E} \{\epsilon^T\epsilon^*-\Psi(\epsilon) \} \quad \forall \epsilon^*\in {\cal E^*} ,
\end{equation}
where  ${\cal E^*}$ is a dual space of $\cal E$.
Clearly, $\Psi^*(\epsilon^*)$ is convex and lower semi-continuous.
 By the theory of convex analysis,
  the following canonical duality relations hold on  $ {\cal E}\times {\cal E^*} $:
\begin{equation}\label{eq: relations}
\epsilon^* \in \partial \Psi(\epsilon) \quad \Leftrightarrow \quad \epsilon \in \partial \Psi^*(\epsilon^*) \quad \Leftrightarrow  \quad \Psi(\epsilon)+ \Psi^*(\epsilon^*)= \epsilon^T\epsilon^*.
\end{equation}
A real-valued function is called the  {\em canonical function} if the canonical duality relations (\ref{eq: relations}) hold.
Based on the standard  canonical dual transformation, we
choose the geometrical operator $\xi_0 = \Lambda_0 (x)= \{ g(x), h(x) \} : \real^n \rightarrow \real^2$ and
let
\[
V_0(\xi_0) = \Psi_1(g) + \Psi_2(h),
\]
where
\begin{equation}
\Psi_1(g)=
\begin{cases}
0 \qquad \mbox{if } g\le0\\
+\infty \quad \mbox{otherwise},
\end{cases}  \;\;\;\;
\Psi_2(h)=
\begin{cases}
0 \qquad \mbox{if } h=0\\
+\infty \quad \mbox{otherwise} ,
\end{cases}
\end{equation}
the constrained problem (\ref{eq: original problem})
can be written in the following canonical form
\begin{equation}\label{eq: indicator f}
\min  \left\{ P(x) = f(x)+ V_0(\Lambda_0(x)) | \;\; \forall x \in \mathbb{R}\right\}.
\end{equation}

By the Fenchel transformation, the conjugate of $V_0(\xi_0)$ can be easily obtained as
$V_0^* (\xi^*_0) = \Psi^*_1(\lambda) + \Psi_2^*(\mu)$, where $\xi_0^* = (\lambda , \mu)$ and
\[
\Psi_1^*(\lambda)=\sup_{g \in \real^m} \{g^T\lambda-\Psi_1(g) \}=
\left\{
\begin{array}{ll}
0 & \mbox{if }  \lambda\ge 0 \\
+ \infty  & \mbox{otherwise} ,
\end{array}
\right.
\]
\[
\Psi_2^*(\mu)=\sup_{h \in \real^p}\{h^T\mu-\Psi_2(h) \} = 0 \quad \forall \quad \mu\in\real^p.
\]
  By using the Fenchel-Young equality $V_0(\xi_0) = \xi_0^T\xi_0^* -  V_0^*(\xi_0^*)$
  to replace $V_0(\Lambda_0(x))$  in  (\ref{eq: indicator f}),
   the
  so called {\em  total complementarity function }  in the canonical duality theory can be obtained in the following form
\begin{equation}\label{eq: complete complementarity}
\Xi_0(x,\lambda, \mu)=f(x)+[\lambda^T g(x)-\Psi_1^*(\lambda)]+[ \mu^T h(x)-\Psi_2^*(\mu)].
\end{equation}

For the indicator $\Psi_1(g)$, the canonical duality relations in  (\ref{eq: relations})
lead to
\begin{equation}\label{eq: conditions on g}
\begin{array}{llll}
\lambda_i \in \partial \Psi_1(g_i) &\Longrightarrow& \lambda\ge 0 & i=1,\dots,m\\
g(x)\in\partial \Psi^*_1(\lambda) &\Longrightarrow & g_i\le0 & i=1,\dots,m\\
\lambda^Tg(x)=\Psi_1(g(x))+\Psi^*_1(\lambda)&\Longrightarrow & \lambda^Tg=0 ,
\end{array}
\end{equation}
which are the KKT  conditions for the inequality constrains $g(x) \le 0$.
While for $\Psi_2(h)$, the same relations in  (\ref{eq: relations})   lead to
\begin{equation}\label{eq: conditions on h}
\begin{array}{llll}
\mu \in \partial \Psi_2(h_j) &\Longrightarrow& \mu\in \real^p\\
h_j(x)\in\partial \Psi^*_2(\lambda) &\Longrightarrow & h_j=0 & j=1,\dots,p\\
\mu^Th(x)=\Psi_2(g(x))+\Psi^*_2(\mu)&\Longrightarrow & \mu^Th=0 .
\end{array}
\end{equation}
From the second and third equation in the (\ref{eq: conditions on h}), it is clear that in order to enforce   the  constrain $h(x)=0$,
 the dual variables $\mu_i$ must be not zero for $i=1,\dots,p$.
 This is a special  complementarity condition for equality constrains,
  generally not mentioned  in many textbooks.
However, the implicit constraint  $\mu \neq 0$
is important in nonconvex optimization. Let $\bsig_0 = (\lambda, \mu) $. The
dual feasible spaces  should be defined as
\[
{\cal S}_0 = \{ \bsig_0 = ( \lambda, \mu)  \in \real^{m \times p }|\;\; \lambda_i\ge 0 \;\; \forall i=1,\dots,m, \;\;
 \mu_j\neq 0  \;\;\forall j=1,\dots,p\}.
\]
Thus, on the feasible space $\real^n\times{\cal S}_0$,
 the total complementary function (\ref{eq: complete complementarity}) can be simplified as
\begin{equation}\label{eq: lagrangian}
\Xi_0(x,\bsig_0)=f(x)+\lambda^T g(x) +\mu^T h(x)={\cal L}(x,\lambda, \mu) ,
\end{equation}
which is   the classical Lagrangian form, and we have
\[
P(x) = \sup \left\{ \Xi_0(x, \bsig_0) | \; \forall \bsig_0  \in {\cal S}_0 \right\} .
\]
This shows that the canonical duality theory is an extension of the Lagrangian theory
(actually, the total complementary function was called the extended Lagrangian in \cite{GaoBook 2000}).
 With  the canonical duality theory it is possible to formulate the optimality conditions for both
  inequality and equality constraints
  in an unified framework.

If  $f$,  $g$ are convex and   $h$ is linear, the Lagrangian (\ref{eq: lagrangian}) is a saddle function, i.e.
${\cal L}(x,\lambda,\mu)$  is convex in the primal variable $x$ and concave(linear) in the dual variables $\lambda$ and $\mu$.
In this case, the Lagrangian dual can be defined by
$$
P^*(\lambda, \mu)= \inf_{x\in{\cal X}_a} {\cal L}(x, \lambda, \mu)
$$
on a subspace ${\cal S}_a \subset {\cal S}_0$ and
the saddle Lagrangian duality leads to   the following strong duality relation
$$
\inf_{x\in{\cal X}_a} {\cal L}(x, \lambda, \mu)=\sup_{(\lambda, \mu)  \in{\cal S}_a} P^*(\lambda, \mu).
$$
It is well-known that this Lagrangian duality holds only for convex problems.
For general nonconvex constrained problems,
only the weak duality relation is available, i.e. there is a duality gap between the primal problem and its
Lagrangian dual.
 With the canonical duality theory, it is possible to  close the duality gap to obtain global optimal solutions.

\section{ Sequential Transformation for Nonconvex Problems}
In order to solve nonconvex constrained problems in a unified way,
 the nonconvex functions should be assumed to have certain patterns in order to model real-world problems.
In this paper, we need the following assumption.
\begin{assumption}\label{bounded below}
The nonconvex functions
$f$, $g_i$ and $h_j$ for $i=1,\dots, m$ and $j=1,\dots,p$  can be expressed in the following way:
\begin{equation}
\begin{array}{l}
f(x)=  V_f(\Lambda_f (x))+\frac{1}{2} x^T Ax-c^Tx\nonumber\\
g_i(x)=  V_{g_i}(\Lambda_{g_i} (x))\quad i=1,\dots, m\\
h_j(x)= V_{h_j}(\Lambda_{h_j} (x)) \quad j=1,\dots, p
\end{array}
\end{equation}
where $\xi_f=\Lambda_f (x)$, $\xi_{g_i}=\Lambda_{g_i} (x)$ and $\xi_{h_j}=\Lambda_{h_j} (x)$ are
quadratic geometrical operators such that $V_f(\xi_f)$,
$V_{g_i}(\xi_{g_i})$, $V_{h_j}(\xi_{h_j})$ are differentiable canonical functions for every
 $i=1,\dots, m$ and $j=1,\dots, p$.
\end{assumption}

Based on this assumption, we can define the following second-level  geometrical operators
\[
\xi_g= \Lambda_g (x) = \left(\Lambda_{g_i},\dots, \Lambda_{g_m} \right) , \;\;  \quad \xi_h= \Lambda_g (x)
 = \left(\Lambda_{h_j},\dots, \Lambda_{h_p} \right).
\]
Let $\xi_1 = (\xi_f , \xi_g, \xi_h)= \Lambda_1(x)$, $\VV_g (\xi_g)  = \{ \VV_{g_i}(\xi_{g_i}) \}$,
and $\VV_h (\xi_h)  = \{ \VV_{h_i}(\xi_{h_i}) \}$. By Assumption 1, the following duality relations are
invertible on their domains, respectively,
\begin{equation}\label{eq: dual map}
 \xi^*_f  =  \nabla V_f(\xi_f) , \;\;  \xi^*_g =  \nabla V_g(\xi_g) ,  \;\;
 \xi^*_h = \nabla V_h(\xi_h).
\end{equation}
Also, the Legendre conjugates $\VV^*_f (\xi_f^*), \;\; \VV^*_g(\xi^*_g)$ and $\VV^*_h (\xi_h^*)$ can be defined uniquely.

 Denote $\bsig_1 = (\sigma_f, \sigma_g, \sigma_h) = ( \xi^*_f, \xi^*_g, \xi^*_h) $ and let ${\cal S}_1$ be a domain such that on which,
 the inverse
 duality relations (\ref{eq: dual map}) hold.
 By using    the Fenchel-Young equalities, the first-level total complementary function $\Xi_0$  (\ref{eq: lagrangian}) can be
written in the following second-level form:
\begin{eqnarray} \label{eq: tot comp1}
  \Xi_1(x,\bsig_0,\bsig_1 ) &=& \Lambda_f(x) \sigma_f - V^*_f(\sigma_f)+\lambda^T( \Lambda_g(x)\circ \sigma_g - V^*_g(\sigma_g))  \nonumber \\
& & + \mu^T( \Lambda_h(x) \circ \sigma_h - V^*_h(\sigma_h))-U(x),
\end{eqnarray}
where   $U(x)=c^Tx- \frac{1}{2}x^TAx$,
 and the symbol $\circ$ indicates the Hadamard product  between the primal and dual variables, i.e.,
$$
 \xi_g \circ \sigma_g = \left(\xi_{g_1} \sigma_{g_1}, \dots, \xi_{g_m} \sigma_{g_m} \right) .
$$
Based on (\ref{eq: tot comp1}), the canonical dual function can be obtained by
\begin{equation}\label{eq: dual}
P^d(\bsig_0,\bsig_1)= U^\Lambda(\lambda, \mu, \sigma)-\left(V_f^*(\sigma_f)+\lambda^TV_g^*(\sigma_g)+\mu^TV_h^*(\sigma_h)\right) ,
\end{equation}
where $U^\Lambda(\bsig_0,\bsig_1)$ is the  $\Lambda$-conjugate of  $U(x)$ defined by (see \cite{GaoBook 2000})
\begin{equation}\label{eq: u-stat}
U^\Lambda(\bsig_0,\bsig_1)=\sta\{ \Lambda_f(x) \sigma_f +\lambda^T\left( \Lambda_g(x)\circ \sigma_g \right)+\mu^T
\left( \Lambda_h(x) \circ \sigma_h \right)-U(x): x\in \real^n
 \}
\end{equation}
 Let ${\cal S}_a \subset {\cal S}_0 \times {\cal S}_1$ be
  the canonical dual feasible space such that on which,  $U^\Lambda(\bsig_0,\bsig_1)$ is well-defined.
The canonical dual problem can be proposed as
$$
({\cal P}^d): \quad \sta\{P^d(\bsig_0,\bsig_1): \;\;  (\bsig_0,\bsig_1)\in {\cal S}_a  \}.
$$
\begin{figure}[h]
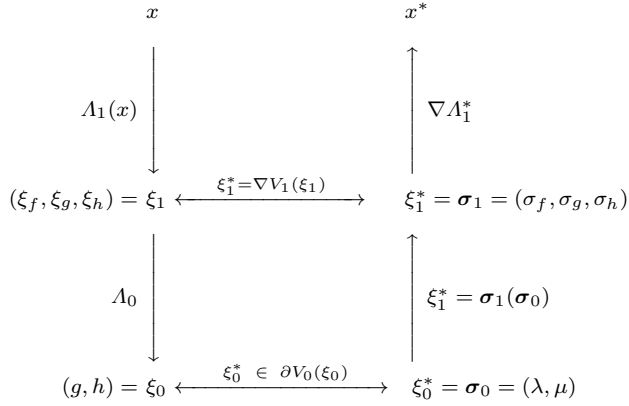

{
$$
\begin{array}{rlclc}
&x& &x^*& \\[0.2cm]
\Lambda_1(x) &\left\downarrow\rule{0cm}{1.cm}\right.\phantom{(\varphi^n)^*}& &
\left\uparrow\rule{0cm}{1.cm}\right.
\nabla\Lambda^*_1\\[0.2cm]
(\xi_f, \xi_g, \xi_h) = & \xi_1 \xleftrightarrow{\phantom{xxx}{\xi^*_1 = \nabla V_1(\xi_1)}\phantom{xxx}}&&
\xi^*_1 =\bsig_1 = (\sigma_f, \sigma_g, \sigma_h) &\\[0.2cm]
\Lambda_0  &\left\downarrow\rule{0cm}{1.cm}\right.\phantom{(\varphi^n)^*}& &
\left\uparrow\rule{0cm}{1.cm}\right.  \xi_1^* = \bsig_1 (\bsig_0)  \\[0.2cm]
(g, h) = & \xi_0 \xleftrightarrow{\phantom{x}{\;\;\;\;\; \xi^*_0 \;\;\in \;\;\partial V_0 (\xi_0)\;\;\;\;}\phantom{x}}
&&\ \xi^*_0 = \bsig_0 =(\lambda, \mu)  &\\[0.2cm]
\end{array}
$$
}
\caption{The scheme of the  sequential canonical dual transformation.}
   \label{fig: sequential}
\end{figure}

\begin{theorem}\label{com dual prin}
(Complementarity Dual Principle) Suppose that the point $(\bar{x},\bar{\bsig}_0,\bar{\bsig}_1)$ is a
critical point for the total complementarity function (\ref{eq: tot comp1}), then   $\bar{x}$
 is a KKT point of the primal problem (\ref{eq: original problem}),
  $(\bar{\bsig}_0,\bar{\bsig}_1)$ is a KKT point of the dual problem (\ref{eq: dual}) and
$$
P(\bar{x})=\Xi_1(\bar{x},\bar{\bsig}_0,\bar{\bsig}_1)=P^d(\bar{\bsig}_0,\bar{\bsig}_1)
$$
\end{theorem}
\begin{proof}.
If $(\bar{x},\bar{\bsig}_0,\bar{\bsig}_1)$ is a critical point for the total complementarity function (\ref{eq: tot comp1}) then it must satisfy the following first order conditions:
\begin{eqnarray}\label{eq: foc tot com}
&\nabla_x\Xi_1(\bar{x},\bar{\bsig}_0,\bar{\bsig}_1)=&
\nabla\Lambda_f (\bar{x})\bas_f+\bal^T(\nabla\Lambda_g (\bax)\circ \bas_g)\nonumber\\
&&+\bam^T(\nabla\Lambda_h(\bax)\circ\bas_h)+Ax-c = 0 , \nonumber\\
&\nabla_{\sigma_f}\Xi_1(\bar{x},\bar{\bsig}_0,\bar{\bsig}_1)=& \Lambda_f(\bax)-\nabla V^*_f(\bas_f) = 0  ,\nonumber\\
&\nabla_{\sigma_g}\Xi_1(\bar{x},\bar{\bsig}_0,\bar{\bsig}_1)=& \Lambda_g(\bax)-\nabla V^*_g(\bas_h) = 0 , \nonumber\\
&\nabla_{\sigma_h}\Xi_1(\bar{x},\bar{\bsig}_0,\bar{\bsig}_1)=& \Lambda_h(\bax)-\nabla V^*_h(\bas_h) = 0 .
\end{eqnarray}
The last three conditions in the (\ref{eq: foc tot com}) are equivalent to
$$
\bas_f=\nabla V_f\left(\Lambda_f (\bax)\right), \quad \bas_g=
 \nabla V_g\left(\Lambda_g(\bax)\right), \quad \bas_h= \nabla V_h \left(\Lambda_f(\bax)\right).
$$
By substituting these conditions in the first equation of the (\ref{eq: foc tot com})
and using the chain rule of derivation on $f$, $g_i$ and $h_j$ for every $i=1,\dots, m$ and $j=1,\dots, p$,
we obtain
$$
\nabla f(\bax)+ \bal^T \nabla g(\bax) +\bam^T \nabla h(\bax)=\nabla {\cal L} (\bax,\bal,\bam)=0.
$$
This condition plus the conditions coming from the (\ref{eq: conditions on g}) prove that
 $\bax$ is a KKT point for the (\ref{eq: original problem}).
  Furthermore, from these complementarity conditions we obtain that
  $f(\bar{x})=\Xi_1(\bar{x},\bar{\bsig}_0,\bar{\bsig}_1)$.

The first equation of the (\ref{eq: foc tot com}) leads to the satisfaction of the stationarity condition
 (\ref{eq: u-stat}) that is:
\[
U^\Lambda(\bsig_0, \bsig_1)= \Lambda_f(x) \sigma_f +\lambda^T\left( \Lambda_g(x)\circ \sigma_g \right)+
\mu^T\left( \Lambda_h(x) \circ \sigma_h \right)-U(x).
\]
This together with the property that the first order conditions of the dual are equivalent to the
 last three conditions of the (\ref{eq: foc tot com}) proves that $(\bar{\bsig}_0,\bar{\bsig}_1) = (\bar{\lambda},\bar{\mu},\bar{\sigma})$ is
  a KKT point of the dual and
$\Xi_1(\bar{x},\bar{\bsig}_0,\bar{\bsig}_1)=P^d(\bar{\bsig}_0,\bar{\bsig}_1)$.
\qed
\end{proof}
This theorem shows  that with the canonical duality theory and the  sequential canonical dual transformation it is
possible to close the duality gap between the nonconvex  primal problem and its canonical  dual problem.

\section{Global Optimality Solutions}\label{lag num ex}

In order to have conditions for the global minimum of the original constrained problem
(\ref{eq: original problem}), we make the following assumptions
\begin{assumption}\label{convexity}
The canonical functions $V_f(\xi_f)$, $V_{g_i}(\xi_{g_i})$, and $V_{h_j}(\xi_{h_j})$
 are convex    for all $i=1,\dots,m$ and $j=1,\dots,p$.
 Furthermore,  for any Lagrange multiplier $\mu\in\real^p$, we assume that
\[
\mu^Th(x)>-\infty \quad \forall x\in \real^n.
\]
\end{assumption}

 Since $\Xi_1$ is a quadratic function of $x$, its  Hessian matrix is $x$-free and can be defined by
  $G( {\bsig}_0, {\bsig}_1)=\nabla_x^2\Xi_1( {\bsig}_0, {\bsig}_1)$. Let
\begin{equation}
{\cal S}_a^+=\{( {\bsig}_0, {\bsig}_1)\in {\cal S}_a | \;\;
 G( {\bsig}_0, {\bsig}_1)\succ 0, \;\; \mu_i > 0 \;\; \forall i=1, \dots, p\} .
\end{equation}
\begin{theorem}\label{global condition}
(Global Optimality Conditions) Suppose that Assumptions \ref{bounded below} and
\ref{convexity} are satisfied, and ${\cal S}_a^+ $ is convex. Then if the point $(\bar{x},\bar{\bsig}_0, \bar{\bsig}_1)$
is a critical point of the $\Xi_1$ and $(\bar{\bsig}_0, \bar{\bsig}_1)\in {\cal S}_a^+$,
then $(\bar{\bsig}_0, \bar{\bsig}_1)$ is the global maximizer of $P^d$ on
${\cal S}_a^+$ and $\bax$ is the global minimizer of $P$ on ${\cal X}_a$, that is
\[
P(\bax)=\min_{x\in{\cal X}_a} P(x)=\max_{( {\bsig}_0, {\bsig}_1)\in {\cal S}_a^+}
P^d({\bsig}_0, {\bsig}_1)=P^d(\bar{\bsig}_0, \bar{\bsig}_1 )
\]
\end{theorem}
\begin{proof}.
By Assumption \ref{convexity}, the functions $V_f(\xi_f)$, $V_g(\xi_g) $ and $V_h(\xi_h)$
are convex. This implies that their  Legendre conjugates are  also convex.
 Because of the  positivity of both $\lambda$ and $\mu$,
 the total complementarity function $\Xi_1$ is concave in the dual variables $\sigma_f$, $\sigma_g$ and $\sigma_h$.
 Also these variables are decoupled. This implies that the following relation
\[
\max_{\sigma_f}\max_{\sigma_g}\max_{\sigma_h} \Xi_1(x,\bsig_0,\bsig_1)=\max_{\bsig_1}\Xi_1(x,\bsig_0,\bsig_1),
\]
is always verified in ${\cal S}_1$. By the fact that $\Xi_0$ is linear in both $\lambda$ and $\mu$ we have
\[
\max_{(\lambda,\mu)\in {\cal S}_0}\max_{\bsig_1\in {\cal S}_1}\Xi_1(x,\lambda,\mu,\bsig_1)
=\max_{(\lambda,\mu)\in {\cal S}_0}  {\cal L}(x,\lambda,\mu)=
\begin{cases}
\begin{array}{ll}
P(x)&    \mbox{if } \quad x\in{\cal X}_a\\
\infty &  \mbox{otherwise}.
\end{array}
\end{cases}
\]
Furthermore if $(\bsig_0,\bsig_1) \in {\cal S}_a^+$, then the total complementarity function is convex in $x$
 and concave in $\bsig_1$. For this reason the $\min$ and $\max$ statements can be exchanged in the
 total complementarity function and we obtain
\begin{eqnarray}
 \min_{x\in{\cal X}_a} P(x)&= \displaystyle\min_{x\in\real^n}\max_{(\bsig_0, \bsig_1)\in {\cal S}_a^+}
 \Xi_1(x,\bsig_0, \bsig_1) &\nonumber\\
&=\displaystyle\max_{( \bsig_0, \bsig_1 )\in {\cal S}_a^+}\min_{x\in\real^n}
\Xi_1(x,\bsig_0, \bsig_1  )=& \max_{(\bsig_0, \bsig_1  )\in {\cal S}_a^+} P^d(\bsig_0, \bsig_1 ).
\end{eqnarray}
This proves the theorem. \qed
\end{proof}

\begin{remark}
Since the geometrical operator  $\Lambda_1(\bx)$ is a quadratic vector-valued function of $\bx$,
 by  Assumption 1,
the canonical dual function $P^d(\bsig_0, \bsig_1)$ can be written in the following standard form:
\begin{equation}
P^d(\bsig_0, \bsig_1) =
- \frac{1}{2} \bF(\bsig_0, \bsig_1) \GG^{-1}(\bsig_0, \bsig_1) \bF(\bsig_0, \bsig_1) -  V^*(\bsig_0, \bsig_1) ,
\end{equation}
 where $V^*(\bsig_0, \bsig_1) = \left(V_f^*(\sigma_f)+\lambda^TV_g^*(\sigma_g)+\mu^TV_h^*(\sigma_h)\right)$,
 and  $\bF(\bsig_0, \bsig_1) \in \real^n$ depends on the linear terms in $\Lambda_1(\bx)$ and in $U(\bx)$ (see, for example the  Eqn (81) in \cite{GSR2009}).
By the fact that the canonical dual variables $\bsig_0$ and $\bsig_1$ are generally not independent (see Eqn (4.26) in \cite{GaoBook 2000}),
  even if $P^d(\bsig_0, \bsig_1)$ is concave in $\bsig_0$ and $\bsig_1$ respectively,
it may not be concave in $(\bsig_0, \bsig_1)$ on ${\cal S}_a^+$.
Detailed studies  on the convexity of $P^d(\bsig_0, \bsig_1)$
for polynomial optimization and neural network problems have been discussed in
\cite{gaot,LaG 13}

 \end{remark}

\begin{remark}\label{inverse global solution}
Similarly to Theorem \ref{global condition}, it is possible to  find global maximum conditions by defining
$$
{\cal S}_a^-=\{( {\bsig}_0, {\bsig}_1)\in {\cal S}_a | \;\;
 G( {\bsig}_0, {\bsig}_1)\prec 0, \;\; \mu_i < 0 \;\; \forall i=1, \dots, p\} .$$
Thus, if   $(\bar{x},\bar{\bsig}_0,\bar{\bsig}_1)$ is a critical point of the function $\Xi_1$ and such that
$(\bar{\bsig}_0,\bar{\bsig}_1)$ is the global minimizer of $P^d$ in ${\cal S}_a^-$,
then $\bax$ is the biggest local  maximizer of $P$ on ${\cal X}_a$.

In particular, if the problem is only composed of a quadratic objective function and equality constraints, it is possible to put together these conditions in order to find both the global minimum and global maximum.
\end{remark}

{\bf Example 1}.  Let us consider the following one-dimensional constrained problem
\begin{equation}\label{eq: primal one dim}
\min \left\{ \frac{1}{2}q x^2-cx \;\; | \;\;
s.t. \quad \frac{1}{2}\left(\frac{1}{2} x^2-d \right)^2-e=0 \right\}.
\end{equation}
Since the constraint $h(x)$ is a fourth-order polynomial (double well function),
we let $\Lambda_h(x) = \frac{1}{2} x^2$, the canonical dual  function can be obtained as
\[
P^d(\mu,\sigma)=-\frac{c^2}{2\left(q+\mu\sigma\right)}-\mu\left(\frac{1}{2}\sigma^2+\sigma d+e \right).
\]
In this particular example with  only one equality constraint, we have $\bsig_0 = \mu, \;\; \lambda = 0$ and
$\bsig_1 = \sigma $.
If we let   $q=1$, $c=1$, $d=6$, $e=15$, there are total four
 KKT points as reported in Table \ref{critical point ex 1}.
\begin{table}[ht]
\begin{center}
\begin{tabular}{c| c c c c c c}
\hline
&$x$  &$\mu$& $\sigma$& $f(x)$ &$P^d(\mu,\sigma)$&$G(\mu,\sigma)$\\
\hline
$(x_1,\mu_1, \sigma_1)$&$1.023$& $0.004$&$-5.48$&$-0.5$&$-0.5$&$0.98$\\
$(x_2,\mu_2, \sigma_2)$&$-1.023$& $0.36$&$-5.48$&$1.55$&$1.55$&$-0.98$\\
$(x_3,\mu_3, \sigma_3)$&$4.791$ &$ -0.14$&$5.48$&$6.69$&$6.69$&$0.21$\\
$(x_4,\mu_4, \sigma_4)$&$-4.791$&$-0.22$&$5.48$&$16.27$&$16.27$&$-0.21$\\
\hline
\end{tabular}
\end{center}
\caption{Critical points of the primal and dual problems for example (\ref{eq: primal one dim}) with $q=1$, $c=1$, $d=6$, $e=15$.}
\label{critical point ex 1}
\end{table}
It is easy to see  that there is no duality gap between the solutions of the primal and the dual problems just as
reported in Theorem \ref{com dual prin}. From the values of the multipliers $\mu$
 at the optimum, we can say that the first two critical points are the solutions of the minimization problem,
  while the last two are the solutions for the maximization problem. If we check to which domain the solutions belong, we have that
   $(\mu_1, \sigma_1)$ is  the global maximum in ${\cal S}_a^+$ while $(\mu_4, \sigma_4)$ is the local  minimum in ${\cal S}_a^-$.
    This means that $x_1$ is the global minimum of the original constrained problem, while $x_4$ is the biggest local
     maximum of the original constrained problem.
    This example shows once again that thanks to canonical duality theory, not only we are able to close the gap created by dropping the convexity
    assumptions in the Lagrangian function, but we are also able to obtain the conditions for finding the global minimum.

\section{Augmented Lagrangian}
We want to compare the approach of the Lagrangian with the one of augmented Lagrangian by using canonical duality theory.
We will consider the problem  only with one non-convex equality constraint $h(\bx) = 0$ (i.e. $p = 1$):
\begin{equation}\label{eq: aug lag}
{\cal L}_\nu(x,\mu;\nu)=f(x)+\mu  h(x)+\frac{1}{2\nu}\|h(x)\|^2,
\end{equation}
Where $\nu$ is a penalty parameter. The principal framework of Augmented Lagrangian consists in solving a
sequence of sub-problems with both the penalty constant $\nu_k$ and the Lagrangian multiplier $\mu_k$ fixed.
 At each iteration, a local minimum in $x$ of the function (\ref{eq: aug lag}) with fixed $\mu_k$ is found.
 The penalty constant is generally updated by $\nu_{k+1}=\alpha\nu_k$ with $\alpha\in (0,1)$, while the multipliers are updated in the following way:
\begin{equation}\label{eq: lam update}
\mu_{k+1}=\mu_k+\frac{h(x)}{\nu}.
\end{equation}
Then a new sub-problem with updated parameters is generated and a new iteration begins.

We analyze both the general case in which $\mu$ is considered as variable and the sub-problem in which $\mu_k$ is fixed.
Differently
 from the augmented Lagrangian approach, with canonical duality theory it is possible to consider $\mu$ as a variable.

Similarly with the previous sections we make the assumption that every equality constraint can be written in the following way
$$
h (x)=V_{h}(\xi_{h })=V_{h }(\Lambda_{h } (x))  ,
$$
where $V_{h }$ is  convex canonical function and $\Lambda_h $ is  a quadratic operator. The augmented Lagrangian can be written as:
\begin{equation}\label{eq: aug lag2}
{\cal L}_\nu(x,\mu)=f(x)+\mu  V_{h} (\Lambda_h(x))+\frac{1}{2\nu}\|V_{h} (\Lambda_h(x))\|^2,
\end{equation}
This function is different than the Lagrangian, as the penalty term adds a further level of complexity,  but with a simple canonical transformation we  can go back to a form similar to the (\ref{eq: lagrangian}). We choose as non-linear operator $\xi_0=h(x)$ and by following the same procedure for canonical duality transformation in the previous sections we obtain:
\begin{equation}\label{eq: aug dual var}
V_0(\xi_0)= \frac{1}{2\nu}\xi_0^2,
\;\;
 \quad \tau=\nabla V_0(\xi_0)=\frac{\xi_0}{\nu}, \;\;
 \quad V^*_0(\tau)=\frac{\tau^2\nu}{2}.
\end{equation}
It is important to notice that the dual variable $\tau$ at the optimum has the value of the increment that should be applied to $\mu$
at every iteration as described in the (\ref{eq: lam update}). By using the Fenchel-Young equality we obtain:
\begin{equation}\label{eq: xi aug}
\Xi_0^\nu(x,\mu,\tau)=f(x)+(\mu+\tau)^TV_{h} (\Lambda_h (x))-V^*_0(\tau)
\end{equation}
This formula is similar in its structure to the (\ref{eq: lagrangian}). By looking at the (\ref{eq: xi aug}), it is clear that
 because of the assumptions made on the constrains $h(x)$, the quantity $(\mu+\tau)$ must be positive in order to ensure that $\Xi_0^\nu(x,\mu,\tau)$ is
 bounded below in $x$. Furthermore, the quantity $(\mu+\tau)$ must not be zero otherwise the constrain would be ignored.
 By using the same procedure showed in the previous section we obtain:
\begin{equation}\label{eq: tot com aug}
\Xi_1^\nu(x,\mu,\tau,\sigma)= \Lambda_f(x) \sigma_f - V^*_f(\sigma_f)+(\mu+\tau)^T( \Lambda_h(x)  \sigma_h - V^*_h(\sigma_h))-V^*_0(\tau)-U(x)
\end{equation}
and the dual formulation is:
\begin{equation}\label{eq: aug dual}
P^{d}_\nu(\mu,\tau, \sigma)= U^\Lambda( \mu, \sigma)-\left(V_f^*(\sigma_f)+(\mu+\tau)^TV^*_h(\sigma_h)+\frac{\tau^2\nu}{2}\right)
\end{equation}
\begin{remark}
The complementary-dual principle proved in Theorem
 \ref{com dual prin} for the Lagrangian function can be easily extended to the
 critical points of ${\cal L}_a^\nu(x,\mu)$ and $P^{d}_\nu(\mu,\tau, \sigma)$ as well.
\end{remark}

\begin{theorem}\label{correspondence}
If $(\bam,\bat, \bas)$ is a critical point for $P^{d}_\nu(\mu,\tau, \sigma)$, then $\bat=0$.
Furthermore we have
$$
P^d_{\nu}(\bam,0,\bas)=P^d(\bam,\bas),
$$
that is $P^d_{\nu}$ and $P^d$ are equivalent in their stationary points and Theorem \ref{global condition} can be applied to find the global minimum.
\end{theorem}
\begin{proof}
From the second conditions in the (\ref{eq: aug dual var}) we have that in critical points $\bat=\frac{h(\bax)}{\nu}$. As $(\bax,\bam)$ is
a feasible KKT point with associated multipliers $\bam$, we have that $\frac{h(\bax)}{\nu}=0$.
 If $\bat$ is zero for every critical point, then by plugging this value in every $\tau$ of the (\ref{eq: aug dual}) we obtain the (\ref{eq: dual}).
\qed
\end{proof}
\begin{remark}
Theorem  \ref{correspondence} shows that,
from canonical duality point of view,
the use of the penalty term is not necessary in the problems considered in this paper because
 it increases both the complexity of the primal problem and
the dimensionality of the dual problem. By solving the dual problem in both the Lagrange multiplier $\mu$ and dual variable $\sigma$
it is possible to find the global solution of the original problem.
\end{remark}

\subsection{Solution to the Sub-Problem}
Like we have stated in the previous section, the strategy of the augmented Lagrangian creates a succession of
sub-problems with solutions are convergent to a stationary point of ${\cal L}(x,\mu)$.
In these sub-problems both $\mu$ and $\nu$ are fixed to certain values and then updated once the sub-problem
is solved and before a new iteration starts. In this section we want to apply canonical duality theory to the subproblem.
 The primal problem is
$$
{\cal L}_{\nu,\mu_k}(x)=f(x)+\mu_k  h(x)+\frac{1}{2\nu}\|h(x)\|^2,
$$
with associated dual similar to the (\ref{eq: aug dual}), that is
$$
P^{d}_{\nu,\mu_k}(\tau, \sigma)= U^\Lambda( \mu_k, \sigma)-\left(V_f^*(\sigma_f)+(\mu_k+\tau)^TV^*_h(\sigma_h)+\frac{\tau^2\nu}{2}\right)
$$
We also define the following matrix:
$$
G(\tau,\sigma)= \nabla^2_\bx  \Xi_1^{\nu,\mu_k}( \bx, \tau,\sigma),
$$
where $\Xi_1^{\nu,\mu_k}(x,\tau,\sigma)$ is the total complementarity function that connects the primal and
 dual problem that can be easily obtained by the (\ref{eq: tot com aug}). Let
\begin{equation}
{\cal S}_{a,\mu_k}^+=\{(\tau, \sigma)\in {\cal S}_a |  \;\;   G(\tau, \sigma) \succ  0 \}.
\end{equation}

In this case the solution of the sub-problem ${\cal L}_\nu(x,\mu_k)$ are not KKT points of the original problem
(\ref{eq: original problem}) and Theorem \ref{correspondence} cannot be applied due to the additional penalty term. By
the canonical duality we have the following Corollary.
\begin{corollary}\label{aug coro}
Suppose that the point $\bax$ is a stationary point of ${\cal L}_{\nu,\mu_k}(x)$, then $\bax$ has a corresponding $(\bat,\bas)$ that is a stationary point of the $P^{d}_{\nu,\mu_k}$ and
$$
{\cal L}_{\nu,\mu_k}(\bax)=P^{d}_{\nu,\mu_k}(\bat, \bas).
$$
Furthermore if $\mu_k+\tau>0$ and $(\bat,\bas)\in {\cal S}_a^+$ then $\bax$ is the global minimizer of ${\cal L}_\nu(x,\mu_k)$.
\end{corollary}
\begin{proof}
This proof is similar to those of Theorem \ref{com dual prin} and Theorem \ref{global condition} and can be omitted.
\qed
\end{proof}
Because of this Corollary, it is possible to find the global solution $x^*$ to ${\cal L}_{\nu,\mu_k}(x)$ for any value of $\nu$ and $\mu_k$. Furthermore, as $\tau=\frac{h(x)}{\nu}$, it is possible to update the current value of the multiplier $\mu_{k+1}=\mu_k+\tau^*$, where $\tau^*$ is the dual variable corresponding to $x^*$, to get closer to the Lagrangian multiplier $\mu^*$ of the global solution.

\subsection{Sub-Problem Example}

In this subsection we study the same example already proposed in Section \ref{lag num ex} but with the augmented Lagrangian.
First we show how the penalty term, in the case of non-convex constraints,
 greatly increases the complexity of the problem.
 From Figure \ref{target} it is possible to see the target function and the constrain.
 The black dots in the picture highlight the four KKT points for this problem.
\begin{figure}[h]
\begin{center}
\includegraphics[scale=.3]{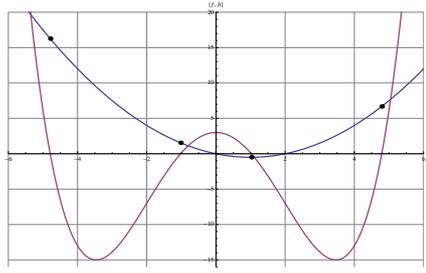}
\caption{Graphs of the target function (blue) and constrain (red) with KKT points highlighted}
\label{target}
\end{center}
\end{figure}
\begin{figure}
 \begin{minipage}{0.45\textwidth}
   \centering
   \includegraphics[scale=.2]{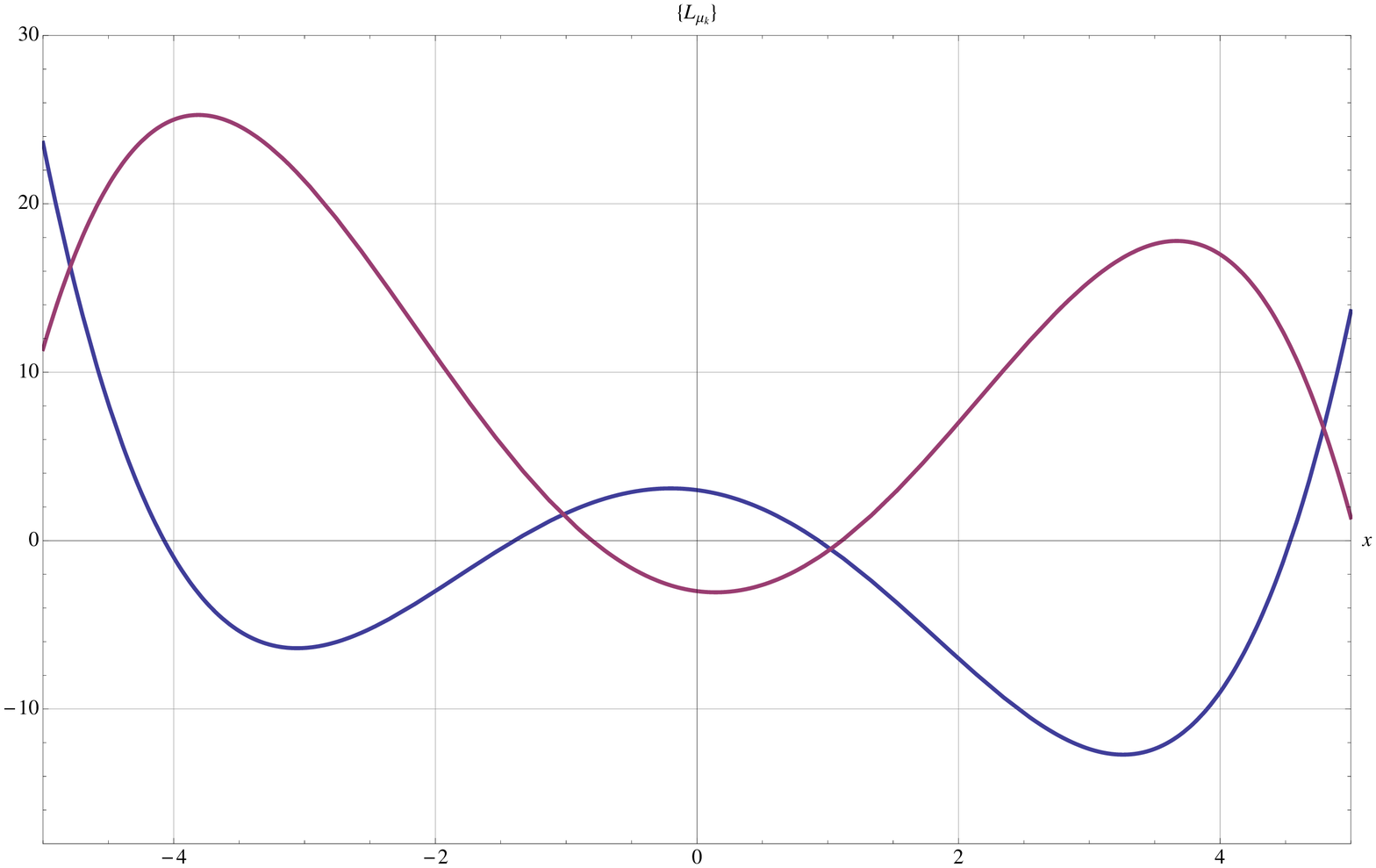}
   \caption{Comparison for the Lagrangian with positive  (blue) and negative multipliers( red)}
   \label{lagra}
 \end{minipage}
 \quad
 \begin{minipage}{0.45\textwidth}
  \centering
   \includegraphics[scale=.33]{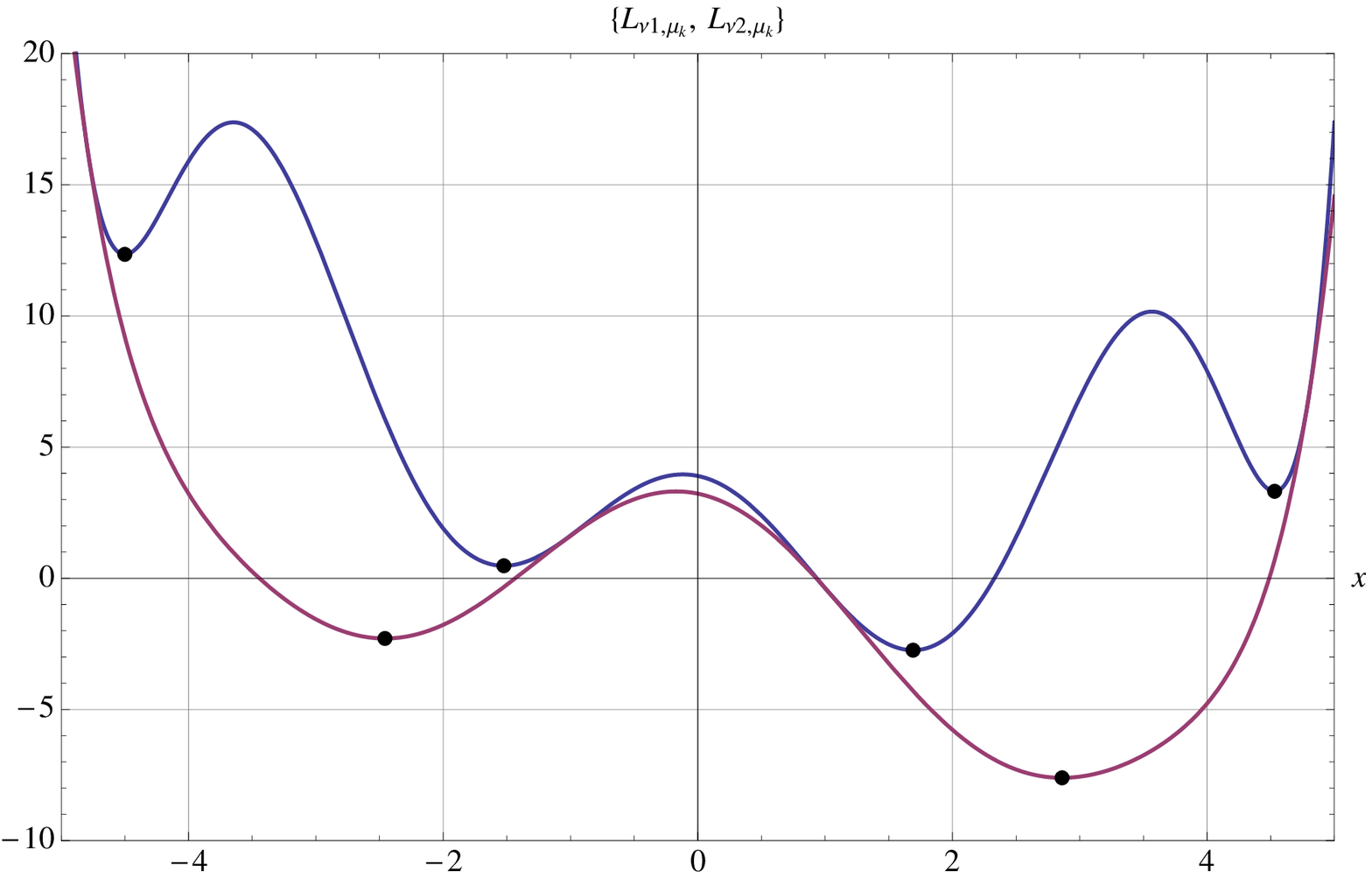}
   \caption{Augmented Lagrangian for  $\mu = 5 $ (blue)  and   $\mu = 20 $ (red).}
   \label{augment}
 \end{minipage}
\end{figure}
Figure \ref{lagra} shows the Lagrangian function for positive multiplier $\mu=1$ and negative multiplier $\mu=-1$.
In both cases we observe the presence of a double well. In the case of positive multiplier there are the two local minima,
while in the case of negative multipliers the two local maxima can be seen.

Finally in Figure \ref{augment} two augmented Lagrangian functions are shown. The blue function has a
relatively smaller  value of the penalty parameter, $\nu=5$,
 while the red function has a big  value of the penalty parameter, $\nu=20$.
  The small values   $\nu$ produce
  nonconvex augmented Lagrangian,
   and the points corresponding to local maxima of the original problem are made into local minima by the penalty term.
   This produces much more difficulties in numerical computation for finding the global optimal solution.

We have already showed in section \ref{lag num ex} that
the  canonical duality theory is able to find the global minimum of the Lagrangian function,
and at the beginning of this section we showed that the same solution is valid if the dual problem of the augmented Lagrangian is solved with
also considering $\mu$ as a variable. Now we show the results for the dual when $\mu_k$ is fixed,
by  Corollary 1  the global solution of the sub-problem can be found.

We solve the problem of the augmented Lagrangian with the same parameters of the problem in Section \ref{lag num ex}, with $\mu=1$ and $\nu=5$.
The function in blue of Figure \ref{augment} is the problem we want to solve. In this case the dual is:
$$
P^{d}_{\nu,\mu_k}(\tau, \sigma)=-\frac{c^2}{2\left(q+(\mu_k+\tau)\sigma\right)}-(\mu_k+\tau)\left(\frac{1}{2}\sigma^2+\sigma d+e \right)-\frac{\tau^2\nu}{2}.
$$
\begin{table}[ht]
\begin{center}
\begin{tabular}{c| c c c c c c c}
\hline
&$x$	&	$\tau$	&	$\sigma$	&	${\cal L}_{\nu,\mu_k}(x)$	&	$P^{d}_{\nu,\mu_k}(\tau, \sigma)$	&	$G(\tau, \sigma)$	&	$(\mu + \tau)$	\\
\hline
$(x_1,\tau_1,\sigma_1)$&1.69	&	-0.91	&	-4.57	&	-2.74	&	-2.74	&	0.59	&	0.09	\\
$(x_2,\tau_2,\sigma_2)$&-1.52	&	-0.66	&	-4.84	&	0.48	&	0.48	&	-0.66	&	0.34	\\
$(x_3,\tau_3,\sigma_3)$&4.53	&	-1.18	&	0.36	&	3.32	&	3.32	&	1.88	&	-0.18	\\
$(x_4,\tau_4,\sigma_4)$&-4.50	&	-1.30	&	4.13	&	12.35	&	12.35	&	-0.22	&	-0.30	\\
$(x_5,\tau_5,\sigma_5)$&-0.12	&	0.59	&	-5.99	&	3.72	&	3.72	&	-8.54	&	1.59	\\
$(x_6,\tau_6,\sigma_6)$&-3.65	&	-2.96	&	0.65	&	17.38	&	17.38	&	-0.27	&	-1.96	\\
$(x_7,\tau_7,\sigma_7)$&3.57	&	-2.99	&	0.36	&	10.16	&	10.16	&	0.28	&	-1.99	\\
\hline
\end{tabular}
\end{center}
\caption{Critical points of the augmented Lagrangian. The first four points correspond to the KKT points of the
original problem, while the last three to the to the maxima of the Lagrangian function}
\label{critical point ex 2}
\end{table}
 Table \ref{critical point ex 2} lists all  critical points of the primal problem and the dual problem.
 From these results we can see  that there is no duality gap between the primal solutions  and their
 canonical dual solutions. By the fact that the point $(x_1,\tau_1,\sigma_1)$  satisfies both
   the conditions: $G(\tau, \sigma) \succeq 0$ and $(\mu + \tau)>0$,
    it is the point corresponding to the global
    minimum of the primal problem, just as it is reported in Corollary \ref{aug coro}.
    Moreover by updating $\mu_{k+1}=\mu_k+\tau_1=0.09$ for the next iteration, the value of the
     multiplier gets closer to the one corresponding to the global minimum, as reported in Table \ref{critical point ex 1}.
     Furthermore, by the conditions in Remark \ref{inverse global solution} adapted for this sub-problem, the point $(x_4,\tau_4,\sigma_4)$ is
      the biggest local maximum of the original problem.

This example shows that even if the problem with non-convex constraints becomes more complicated due to the additional penalty term,
the canonical duality theory is still able to find the global solution.
It is also important to note that for a problem with nonlinear constraints,
 the augmented Lagrangian methods usually produce a  nonconvex sub-problem with double  local minimizers.
Traditional direct methods and algorithms for solving such highly nonconvex  problems have great difficulties to find a good solution.


\section{Conclusions}
In this paper we have shown that the canonical duality theory presents a unified framework to cover traditional
 Lagrangian duality and KKT theory.
 For general nonlinear constrained problems, the popular penalty methods and augmented Lagrangian theory
  may produce nonconvex sub-problems.
 Theorem \ref{correspondence} shows  that as long as the nonconvex constraints satisfy the conditions in Assumption 1 and 2,
the canonical duality theory can be used to solve the problem and the augmented Lagrangian method is indeed  not necessary.

Finally we showed that even with the unnecessary nonconvex term produced by the penalty method,
the  canonical duality theory is still able to find the the best solution of the problem.

\end{document}